\documentclass{amsart}
\usepackage{amsmath}%
\usepackage{amsfonts}%
\usepackage{amssymb}%
\usepackage{graphicx}
\usepackage{color}

\usepackage{hyperref}
\usepackage{enumerate}
\usepackage{eucal}

\newtheorem{theorem}{Theorem}
\newtheorem*{theorem*}{Theorem}
\theoremstyle{plain}

\newtheorem{claim}{Claim}

\newtheorem{conjecture}{Conjecture}
\newtheorem{corollary}{Corollary}

\newtheorem{definition}{Definition}
\newtheorem{example}{Example}

\newtheorem{lemma}{Lemma}
\newtheorem*{lemma*}{Lemma}

\newtheorem{proposition}{Proposition}
\newtheorem{remark}{Remark}

\numberwithin{equation}{section}

\newtheorem{remarks}{Remarks}

\newcommand{\R}{\mathbb{R}}
\newcommand{\Sp}{\mathbb{S}}
\newcommand{\Hy}{\mathbb{H}}

\newcommand{\Span}{\mbox{\normalfont{span}}}

\newcommand{\beeq}{\begin{equation}}
\newcommand{\eneq}{\end{equation}}
\newcommand{\beeqs}{\begin{eqnarray*}}
\newcommand{\eneqs}{\end{eqnarray*}}
\newcommand{\besp}{\begin{split}}
\newcommand{\ensp}{\end{split}}
\newcommand{\bepr}{\begin{proof}}
\newcommand{\enpr}{\end{proof}}
\newcommand{\bethr}{\begin{theorem}}
\newcommand{\enthr}{\end{theorem}}
\newcommand{\beths}{\begin{theorem*}}
\newcommand{\enths}{\end{theorem*}}
\newcommand{\becor}{\begin{corollary}}
\newcommand{\encor}{\end{corollary}}
\newcommand{\bere}{\begin{remark}}
\newcommand{\enre}{\end{remark}}
\newcommand{\bers}{\begin{remarks}}
\newcommand{\enrs}{\end{remarks}}
\newcommand{\beres}{\begin{remark*}}
\newcommand{\enres}{\end{remark*}}
\newcommand{\bele}{\begin{lemma}}
\newcommand{\enle}{\end{lemma}}
\newcommand{\beles}{\begin{lemma*}}
\newcommand{\enles}{\end{lemma*}}
\newcommand{\bepro}{\begin{proposition}}
\newcommand{\enpro}{\end{proposition}}
\newcommand{\bepros}{\begin{proposition*}}
\newcommand{\enpros}{\end{proposition*}}
\newcommand{\becl}{\begin{claim}}
\newcommand{\encl}{\end{claim}}
\newcommand{\beex}{\begin{example}}
\newcommand{\enex}{\end{example}}
\newcommand{\beexs}{\begin{example*}}
\newcommand{\enexs}{\end{example*}}
\newcommand{\beco}{\begin{conjecture}}
\newcommand{\enco}{\end{conjecture}}
\newcommand{\becos}{\begin{conjecture*}}
\newcommand{\encos}{\end{conjecture*}}
\newcommand{\bede}{\begin{definition}}
\newcommand{\bedes}{\begin{definition*}}

\newcommand{\End}{\mbox{\normalfont{End}}}
\newcommand{\lp}{\left(}
\newcommand{\rp}{\right)}

\textheight=\dimexpr
  \ifcase 10pt
    \or 
    \or 
    \or 
    \or 
    \or 
    \or 
    \or 
    \or 
    57\or 
    52\or 
    48\or 
    44\or 
    41\fi 
  \baselineskip+\topskip\relax

\begin{document}
\title[COMPLETE SUBMANIFOLDS WITH RELATIVE NULLITY]{COMPLETE SUBMANIFOLDS WITH RELATIVE NULLITY IN SPACE FORMS}
\author[Canevari, S.]{Samuel Canevari}
\address{Samuel Canevari -- Federal University of Sergipe, Prof. Alberto Carvalho Campus\newline%
\indent Av. Ver. Ol\'{i}mpio Grande - Porto, Itabaiana - SE, 49500-000, Brazil}%
\email{samuel@mat.ufs.br}%

\author[Freitas, G.]{Guilherme Machado de Freitas}
\address{Guilherme Machado de Freitas -- Military Engineering Institute\newline%
\indent Pra\c{c}a Gen. Tib\'{u}rcio, 80 - Urca, Rio de Janeiro - RJ, 22290-270, Brazil}%
\email{guilherme@ime.eb.br}%

\author[Guimar\~{a}es, F.]{Felippe Guimar\~{a}es}
\address{Felippe Guimar\~{a}es -- University of Bras\'{i}lia\newline%
\indent Brasilia, Federal District, 70910-900, Brazil}%
\email{felippe@impa.br}%
\thanks{The third author was supported by by the Coordena\c{c}\~ao de Aperfei\c{c}oamento de Pessoal de N\'ivel Superior - Brasil (CAPES) - Finance Code 001}

\author[Manfio, F.]{Fernando Manfio}
\address{Fernando Manfio -- University of S\~{a}o Paulo\newline%
\indent Av. Trab. S\~ao-carlense, 400, S\~{a}o Carlos - SP, 13566-590, Brazil}%
\email{manfio@icmc.usp.br}%

\author[dos Santos, J. P.]{Jo\~{a}o Paulo dos Santos}
\address{Jo\~{a}o Paulo dos Santos -- University of Bras\'{i}lia\newline%
\indent Brasilia, Federal District, 70910-900, Brazil}%
\email{joaopsantos@unb.br}%
\thanks{The fifth author was supported by FAPDF 0193.001346/2016}

\subjclass[2010]{Primary 53C40; Secondary 53C42} %
\keywords{complete submanifolds, relative nullity, splitting tensor, Milnor's conjecture}%

\begin{abstract}
    We use techniques based on the splitting tensor to explicitly integrate the Codazzi equation along the relative nullity distribution and express the second fundamental form in terms of the Jacobi tensor of the ambient space. This approach allows us to easily recover several important results in the literature on complete submanifolds with relative nullity of the sphere as well as derive new strong consequences in hyperbolic and Euclidean spaces. Among the consequences of our main theorem are results on submanifolds with sufficiently high index of relative nullity, submanifolds with nonpositive extrinsic curvature and submanifolds with integrable relative conullity. We show that no complete submanifold of hyperbolic space with sufficiently high index of relative nullity has extrinsic geometry bounded away from zero. As an application of these results, we derive an interesting corollary for complete submanifolds of hyperbolic space with nonpositive extrinsic curvature and discourse on their relation to Milnor's conjecture about complete surfaces with second fundamental form bounded away from zero. Finally, we also prove that every complete Euclidean submanifold with integrable relative conullity is a cylinder over the relative conullity.
\end{abstract}

\maketitle

\section{Introduction}\label{sec:intro}
The index of relative nullity introduced by Chern-Kuiper \cite{MR0050962} is a fundamental concept in the theory of isometric immersions. At a point $x\in M^n$, the relative nullity subspace $\Delta(x)$ of $f$ is the kernel of the second fundamental form of $f$ at $x$, and the index of relative nullity $\nu(x)$ is the dimension of $\Delta(x)$. It is well known that, on each open subset where the index of relative nullity is a positive constant, the submanifold is foliated by totally geodesic submanifolds of the ambient space. This fact imposes strong restrictions on complete submanifolds of space forms with relative nullity, since the leaves of the minimum relative nullity foliation of a complete submanifold are also complete.

The attempt to understand the obstructions imposed by the presence of relative nullity at any point has been of great interest in submanifold theory over the last four decades. For hypersurfaces in space forms with constant index of relative nullity, there is a useful parametrization introduced by Dajczer-Gromoll \cite{DGGaussParametrization}, the so-called Gauss parametrization, which provides a complete local description of the hypersurface in terms of its Gauss map. In higher codimension, there are several works of both local and global nature on the structure of submanifolds with positive index of relative nullity, e.g., Abe \cite{AbeRiccatiODE}, Maltz \cite{AMHartmanTheorem}, de Freitas-Guimar\~{a}es \cite{GMCylindricity}, and Dajczer, Kasioumis, Savas-Halilaj and Vlachos \cite{DKSVMinimalNullityEuclidean,DKSVMinimalNullitySphere,DKSVMinimalNullityHyperbolic}.

Given a submanifold of a space form with positive index of relative nullity, there is a natural tensor associated to its relative nullity distribution, the so-called \emph{splitting tensor} of the submanifold. The splitting tensor was first introduced by Rosenthal \cite{RDefSplittingTensor} under the name `conullity operator' and was shown to satisfy a Riccati-type ODE by Abe \cite{AbeRiccatiODE}. This equation has proved to be a powerful tool in understanding the behavior of the second fundamental form along geodesics contained in the leaves of relative nullity and has been widely used in the literature. For instance, it is a key ingredient in the proof of Sacksteder's ridigity theorem \cite{Sacksteder} and its generalizations. The solution of the aforementioned ODE allows us to express the splitting tensor of the submanifold explicitly in terms of a Jacobi-type operator $J$ of the ambient space. Moreover, the Codazzi equation can also be explicitly integrated along the relative nullity distribution and as a consequence the shape operators can be written in terms of $J$ too. We point out that, although information on the behaviour of the shape operators $A_\xi$ along the leaves of relative nullity obtained from the differential equation satisfied by them has been used in several works (see Dajczer-Tojeiro \cite{newbook} for instance), the explicit solution of $A_\xi$ in terms of $J$ had proved elusive.

The main goal of this article is to explore the explicit integrability of the Codazzi equation along the relative nullity distribution to derive strong consequences for complete submanifolds of space forms. Among the many achievements of our approach are results on sufficiently high index of relative nullity, nonpositive extrinsic curvature and integrable conullity. In particular, our main result on complete submanifolds with nonpositive extrinsic curvature in a space form yields a complete solution of the natural higher-dimensional version of Milnor's conjecture for hyperbolic ambient space (see Section \ref{sec:milnor}).

\subsection{Sufficiently high index of relative nullity}\label{sec:high}
Among the main applications of the completeness of the leaves of minimum relative nullity foliation is the rigidity of the totally geodesic inclusion of a round sphere $\Sp^n$ into $\Sp^{n+p},\,p\leq n-1$. This is the main result of Ferus \cite{FSphere}, which extends a previous result by O'Neill-Stiel \cite{OSSphere}. Notice that, by the Chern-Kuiper inequality \cite{MR0050962}, we have $\nu^\perp\leq p$ everywhere for an extrinsically flat submanifold of a space form with codimension $p$, where $\nu^\perp=n-\nu$ stands for the \emph{index of relative conullity} of the submanifold. As an immediate consequence of Lemma 32 in Florit-Guimar\~{a}es \cite{belezinha}, we actually have that any isometric immersion $f:M^n\to\Sp^{n+p}$ of a complete Riemannian manifold is totally geodesic provided that its index of minimum relative nullity $\nu_0$ is at least the Radon-Hurwitz number $\rho\lp\nu^\perp_0\rp$ at every point. The situation for complete submanifolds with relative nullity in $\mathbb{H}^{n+p}$ is far more complicated than in $\Sp^{n+p}$. Indeed, even in the simplest case of isometric immersions of $\mathbb{H}^n$ into $\mathbb{H}^{n+1}$, there exist many nontotally geodesic examples, as shown by Ferus \cite{MR0336665}, Nomizu \cite{MR0336664} and Alexander-Portnoy \cite{MR0461379}. However, none of these isometric immersions have \emph{extrinsic geometry} (i.e., second fundamental form) bounded away from zero, and this is a more general phenomenon, according to our main result on complete submanifolds with sufficiently high index of relative nullity in $\mathbb{H}^m$. 

\begin{theorem}\label{cor:hirn}
Let $f:M^n\to\mathbb{H}^m$ be an isometric immersion of a complete Riemannian manifold. If $\nu_0\geq\nu_0^\perp\left(\nu_0^\perp+1\right)/2$, then $f$ does not have extrinsic geometry bounded away from zero.
\end{theorem}

\subsection{Submanifolds with nonpositive extrinsic curvature}
In \cite{BNonPositiveExt}, Borisenko showed that submanifolds with nonpositive extrinsic curvature in low codimension necessarily have a high index of relative nullity. This result was improved by Florit \cite{FNonpositiveExtrinsic} to the sharp inequality $\nu \geq n-2p$, which implies that any isometric immersion $f:M^n\to\Sp^{n+p}$ of a complete Riemannian manifold is totally geodesic provided that $n-\nu_n\geq 2p+1$, where $\nu_n = \max\{k: \rho(n-k) \geq k+1\}$. Using Florit's inequality, we are able to obtain novel results in a similar vein for hyperbolic ambient space.

\begin{theorem} \label{cor:nonpositive-ex}
Let $f:M^n\to\mathbb{H}^{n+p}$ be an isometric immersion of a complete Riemannian manifold with nonpositive extrinsic curvature. If $n\geq 2p^2+3p$, then $f$ does not have extrinsic geometry bounded away from zero. 
\end{theorem}

An interesting application of the above result is the minimality of complete submanifolds of hyperbolic space with nonpositive extrinsic curvature and parallel mean curvature in low codimension. Indeed, we have a slightly more general result.

\begin{corollary}\label{cor:mean-curvature-vector}
Let $f:M^n\to\mathbb{H}^{n+p}$ be an isometric immersion of a complete Riemannian manifold with nonpositive extrinsic curvature and mean curvature vector field with constant length. If $n\geq 2p^2+3p$, then $f$ is minimal. In particular, a complete CMC hypersurface of $\mathbb{H}^{n+1}$ with nonpositive extrinsic curvature is necessarily minimal if $n\geq5$.
\end{corollary}

\subsection{Submanifolds with integrable relative conullity}
Every submanifold with integrable relative conullity $\Delta^\perp$ is locally (globally if it is simply connected and the leaves of $\Delta$ are complete) a generalized cylinder (cf. Tojeiro \cite{MR3439132} and Dajczer-Tojeiro \cite{newbook}). The integrability of $\Delta^\perp$ is equivalent to the fact that the splitting tensor $C_T$ of $\Delta$ is selfadjoint for all $T\in\Gamma(\Delta)$. It turns out that this characterization still holds for any smooth totally geodesic distribution $D\subset\Delta$. In this work, we  take advantage from the latter fact to obtain a complete characterization of cylinders in Euclidean space as well as a universal bound for the extrinsic geometry of the relative conullity leaves in hyperbolic space. We also explicitly state a characterization of totally geodesic submanifolds of the sphere in the same vein which, although following immediately from Lemma 32 in Florit-Guimar\~{a}es \cite{belezinha}, had never been stated before in the literature, as far as we are aware.

\begin{theorem}\label{theo:intcon}
Let $f:M^n\to\mathbb{Q}_c^m$ be an isometric immersion of a complete Riemannian manifold and let $D$ be a smooth totally geodesic distribution of rank $0<k\leq n$ such that $D(x)\subset\Delta(x)$ for all $x\in M^n$. Suppose that the orthogonal complement $D^\perp$ is integrable. Then:
\begin{itemize}
    \item[\emph{(i)}] If $c>0$, we have $k=n$, i.e., $f$ is totally geodesic.
    \item[\emph{(ii)}] If $c=0,\ f$ is a $k$-cylinder.
    \item[\emph{(iii)}] If $c<0$, the inclusion $i:L^{n-k}\to M^n$ of any leaf $L^{n-k}$ of $D^\perp$ has extrinsic geometry bounded by $\sqrt{-c}$.
\end{itemize}
\end{theorem}

\section{Proofs}\label{sec:proof}
The key point to prove all of our main results stated in the introduction is the fact that the Codazzi equation can be explicitly integrated along the relative nullity distribution, which enables us to have a deep understanding on the behavior of the second fundamental form along the leaves of relative nullity. We will start by introducing some definitions and notations that will be used throughout this section.

Let $f:M^n\to\tilde{M}^{n+p}$ be an isometric immersion. We denote by $\alpha$ and $A_\xi$ the second fundamental form and the shape operator of $f$ with respect to a normal vector $\xi$, respectively. The \emph{relative nullity subspace} $\Delta\left(x\right)$ of $f$ at $x$ is the kernel of its second fundamental form $\alpha$ at $x$, that is,
\begin{equation*}
    \Delta \left(x\right)=\left\{X\in T_xM:\alpha\left(X,Y\right)=0\textup{ for all }Y\in T_xM\right\}.
\end{equation*}
The orthogonal complement $\Delta^\perp\lp x\rp$ of $\Delta\lp x\rp$ in $T_xM$ is called the \emph{relative conullity subspace} of $f$ at $x$. The dimensions $\nu\left(x\right)$ and $\nu^\perp\left(x\right)$ of $\Delta\left(x\right)$ and $\Delta^\perp\left(x\right)$ are the \emph{index of relative nullity} and the \emph{index of relative conullity} of $f$ at $x$, respectively. The isometric immersion $f$ is said to be \emph{totally geodesic} at $x\in M^n$ if $\nu\left(x\right)=n$, or equivalently, $\nu^\perp(x)=0$. If $\nu\equiv n$ then $f$ is called a \emph{totally geodesic} isometric immersion.

Let $D$ be a smooth distribution on $M^n$, the \emph{splitting tensor} $C$ of $D$ is the map $C: \Gamma(D)\times\Gamma(D^\perp) \to \Gamma(D^\perp)$ defined by 
\begin{equation*}
    C(T,X) = C_T X = - \left(\nabla_X T\right)_{D^\perp},
\end{equation*}
where $(\ )_{D^\perp}$ stands for the orthogonal projection onto $D^\perp$. It is easily checked that $C$ is tensorial with respect to both variables and hence gives rise to an endomorphism $C_T:D^\perp(x)\to D^\perp(x)$ for any $x\in M^n$ and $T\in D(x)$, which we call the \emph{splitting tensor} of $D$ at $x$ with respect to $T$. 

We are particularly interested in the case in which $D$ is totally geodesic and $D(x)\subset\Delta(x)$ for all $x\in M^n$. In order to integrate the Codazzi equation along $D$, we need the lemma below, which provides useful information on the behavior of $C_T$ and $A_\xi$ along $\Delta$ (a proof can be found in Dajczer-Tojeiro \cite{newbook}).

\begin{lemma}
Let $f:M^n\to\mathbb{Q}_c^m$ be an isometric immersion of a Riemannian manifold and let $D$ be a smooth totally geodesic distribution such that $D(x)\subset\Delta(x)$ for all $x\in M^n$. The differential equation
\begin{equation}\label{eq:spttns}
    \nabla_TC_S=C_SC_T+C_{\nabla_TS}+c\left\langle T,S\right\rangle I
\end{equation}
holds for all $S,T\in\Gamma(D)$. In particular, the operator $C_{\gamma'}$ along a unit speed geodesic $\gamma$ contained in a leaf of $D$ satisfies the differential equation
\begin{equation}\label{eq:stde}
    \frac{D}{dt}C_{\gamma'}=C_{\gamma'}^2+cI.
\end{equation}
Moreover, the differential equation
\begin{equation}\label{eq:sode}
    \frac{D}{dt}A_\xi=A_\xi C_{\gamma'}
\end{equation}
holds for every parallel normal vector field $\xi$ along $\gamma$.
\end{lemma}

Let $\gamma: [0,b) \to M$ be a geodesic such that $\gamma'(0) \in D(\gamma(0))$ contained in a leaf of $D$, the solution of the Jacobi equation $J_{\gamma'(0)}''+cI=0$ along $\gamma$ with initial values $J_{\gamma'(0)}(0)=I,\,J'_{\gamma'(0)}(0)=-C_{\gamma'(0)}$ is given by

\begin{equation}\label{eq:defJacobi}
\begin{aligned}
	J_{\gamma'(0)}(t) &=	\begin{cases}
		\cos{\sqrt{c}t} I - \left(\sin{\sqrt{c}t}/\sqrt{c}\right)C_{\gamma'(0)}, & \text{if } c=1 \\
		\cosh{\sqrt{-c}t} I - \left(\sinh{\sqrt{-c}t}/\sqrt{-c}\right)C_{\gamma'(0)}, & \text{if } c=-1 \\
		I - tC_{\gamma'(0)}, & \text{if } c=0.
		\end{cases}
\end{aligned}
\end{equation} Finally, we denote by $E^{C_{\gamma'(0)}}_{\lambda}$ the eigenspace of $C_{\gamma'(0)}$ associated to the eigenvalue $\lambda \in \R$ and by $\mathcal{P}_0^t$ the parallel transport from $\gamma(0)$ to $\gamma(t)$ along $\gamma$, respectively. We are now in a position to state our main proposition.
    
\begin{proposition}\label{fakemainlemma} 
Let $f:M^n\to\mathbb{Q}_c^m$ be an isometric immersion, and let $D$ be a smooth totally geodesic distribution of rank $0<k<n$ on an open subset $U\subset M^n$ such that $D(x)\subset\Delta(x)$ for all $x\in U$. If $\gamma:\left[0,b\right)\to U$ is a unit speed geodesic such that $\gamma'(0)\in D\left(\gamma(0)\right)$, then $J_{\gamma'(0)}$ is invertible in $\left[0,b\right)$, and the splitting tensor $C_{\gamma'(t)}$ of $D$ along $\gamma$ is given by
\begin{equation}\label{eq:splitting}
    C_{\gamma'(t)}=\mathcal{P}_0^t\circ\left(-J_{\gamma'(0)}'(t)J_{\gamma'(0)}(t)^{-1}\right)\circ\left(\mathcal{P}_0^t\right)^{-1}.
\end{equation}
Moreover, the second fundamental form $\alpha$ of $f$ along $\gamma$ is given by
\begin{equation}\label{eq:secondfundamentalform}
    \alpha_{\gamma(t)}=\mathcal{P}_0^t\circ\alpha_{\gamma(0)}\left(J_{\gamma'(0)}(t)^{-1}\circ\left(\mathcal{P}_0^t\right)^{-1},\left(\mathcal{P}_0^t\right)^{-1}\right).
\end{equation}
Equivalently, the shape operator $A_{\xi_t}$ of $f$ with respect to a parallel normal vector field $\xi_t$ along $\gamma$ is given by
\begin{equation}\label{eq:shapeoperator}
    A_{\xi_t}=\mathcal{P}_0^t\circ A_{\xi_0}J_{\gamma'(0)}(t)^{-1}\circ\left(\mathcal{P}_0^t\right)^{-1}.
\end{equation}
In particular, $\textup{Im}\,A_{\xi_t}$ and $\ker A_{\xi_t}$ are parallel with respect to $\nabla$ along $\gamma$. Therefore, $A_{\xi_t}$ has constant rank along $\gamma$, and hence the numbers of positive and negative eigenvalues remain constant along $\gamma$. Furthermore, $C_{\gamma'(t)}$ is symmetric with respect to $\alpha_{\gamma(t)}$, or equivalently, the endomorphism $A_{\xi_t}C_{\gamma'(t)}$ of $D^\perp\left(\gamma(t)\right)$ is symmetric, and thus $\ker A_{\xi_t}$ is invariant by $C_{\gamma'(t)}$. In addition:
\begin{itemize}
    \item[\emph{(i)}] If $c>0$ and $b\geq\pi/\sqrt{c}$, then $C_{\gamma'(0)}$ can have no real eigenvalues. In particular, $C_{\gamma'(0)}$ cannot be symmetric. Also, for any normal direction at $\gamma(0)$ the numbers of positive and negative principal curvatures are equal.
    \item[\emph{(ii)}] If $c\leq0$ and $\gamma$ extends to a geodesic ray $\tilde{\gamma}:\left[0,\infty\right)\to U$, then all the real eigenvalues of $C_{\gamma'(0)}$ belong to the interval $\left(-\infty,\sqrt{-c}\right]$. Also, for any pair of parallel tangent vector fields $X_t$ and $Y_t$ along $\gamma$ such that $X_0\notin E^{C_{\gamma'(0)}}_{\sqrt{-c}}$ we have that $\alpha_{\gamma(t)}\left(X_t,Y_t\right)\to0$, or equivalently, $A_{\xi_t}X_t\to0$ as $t\to\infty$. In particular, if $\sqrt{-c}$ is not an eigenvalue of $C_{\gamma'(0)}$, then $\alpha_{\gamma(t)}\to0$, or equivalently, $A_{\xi_t}\to0$ as $t\to\infty$. In the case $X_0\in E^{C_{\gamma'(0)}}_{\sqrt{-c}}$ we have that:
    \begin{itemize}
        \item[\emph{(ii.a)}] If $c=0$, then $\alpha_{\gamma(t)}\left(X_t,Y_t\right)$ and $A_{\xi_t}X_t$ are parallel along $\gamma$ with respect to $\nabla^\perp$ and $\nabla$, respectively.
        \item[\emph{(ii.b)}] If $c<0$, and $\alpha_{\gamma(0)}\left(X_0,Y_0\right)\neq0$ or $A_{\xi_0}X_0\neq0$, then $\alpha_{\gamma(t)}\left(X_t,Y_t\right)\to\infty$ and $A_{\xi_t}X_t\to\infty$ as $t\to\infty$, respectively. If $\alpha_{\gamma(0)}\left(X_0,Y_0\right)=0$ or $A_{\xi_0}X_0=0$, then $\alpha_{\gamma(t)}\left(X_t,Y_t\right)\equiv0$ and $A_{\xi_t}X_t\equiv0$, respectively.
    \end{itemize}
Furthermore, if $\gamma$ extends to a geodesic line $\tilde{\gamma}:\left(-\infty,\infty\right)\to U$, we have that:
    \begin{itemize}
        \item[\emph{(ii.1)}] If $c=0$, then the only possible real eigenvalue of $C_{\gamma'(0)}$ is 0. In particular, if $C_{\gamma'(0)}$ is symmetric, then $C_{\gamma'(0)}=0$.
        \item[\emph{(ii.2)}]  If $c<0$, then all the real eigenvalues of $C_{\gamma'(0)}$ belong to the interval $\left[-\sqrt{-c},\sqrt{-c}\right]$. Also, if either $\sqrt{-c}$ or $-\sqrt{-c}$ is not an eigenvalue of $C_{\gamma'(0)}$, then $\alpha_{\gamma(t)}\to0$, or equivalently, $A_{\xi_t}\to0$ as either $t\to\infty$ or $t\to-\infty$, respectively.
    \end{itemize}
\end{itemize}
\end{proposition}

\begin{proof}
Let $T\in\Gamma\left(\gamma^*\left(\End\left(\Delta^\perp\right)\right)\right)$ be the unique solution of the differential equation
\begin{equation}\label{eq:jtfode}
   \frac{D}{dt}T+C_{\gamma'}T=0
\end{equation}
with initial condition $T(0)=I$. Using \eqref{eq:stde} and \eqref{eq:jtfode} we have that $T$ is also a solution on $\left[0,b\right)$ of the second order differential equation with a constant coefficient
\begin{equation}\label{eq:jtsode}
    \frac{D^2}{dt^2}T+cT=0.
\end{equation}
On the other hand, it is straightforward to check that $J(t)=\mathcal{P}_0^t\circ J_{\gamma'(0)}(t)\circ{\left(\mathcal{P}_0^t\right)}^{-1}$ also satisfies \eqref{eq:jtsode}. Since $J(0)=I=T(0)$ and $DJ/dt(0)=-C_{\gamma'(0)}=DT/dt(0)$, we conclude that $T=J$ on $\left[0,b\right)$. Thus, by the uniqueness theorem for solutions of equation \eqref{eq:jtfode}, for each $t\in\left[0,b\right)$ the endomorphism $J(t)$ of $D^\perp\left(\gamma(t)\right)$ is invertible, and \eqref{eq:splitting} holds. Now, let $A(t)=\mathcal{P}_0^t\circ A_{\xi\left(0\right)}J_{\gamma'(0)}(t)^{-1}\circ{\left(\mathcal{P}_0^t\right)}^{-1},\ t\in\left[0,b\right)$. Using \eqref{eq:splitting} we have that $A$ also satisfies \eqref{eq:sode}. Since $A(0)=A_{\xi(0)}$, we conclude that \eqref{eq:shapeoperator}, and equivalently \eqref{eq:secondfundamentalform}, holds. That $A_\xi C_{\gamma'}$ is symmetric follows immediately from equation \eqref{eq:sode}, and the remaining assertions above part (\emph{i}) are direct consequences of equation \eqref{eq:shapeoperator}.

If $c>0$, then the function $\sqrt{c}\cot{\sqrt{c}t}$ takes every real value for $t\in\left[0,\pi/\sqrt{c}\right)$, and since $J_{\gamma'(0)}(t)$ is invertible, this implies that $C_{\gamma'(0)}$ can have no real eigenvalues when $b\geq\pi/\sqrt{c}$. Also, $J_{\gamma'(0)}(t)\to-I$ as $t\to\pi/\sqrt{c}$, so $\left(\mathcal{P}_0^t\right)^{-1}\circ A_{\xi_t}\circ\mathcal{P}_0^t\to-A_{\xi_0}$ as $t\to\pi/\sqrt{c}$, and since the numbers of positive and negative eigenvalues of $A_{\xi_t}$ remain constant along $\gamma$, we conclude that they must be equal.

Similarly, if $c<0$, then the function $\sqrt{-c}\coth{\sqrt{-c}t}$ takes every value in the interval $\left(\sqrt{-c},\infty\right)$ for $t\in\left[0,\infty\right)$, hence all the real eigenvalues of $C_{\gamma'(0)}$ must belong to the interval $\left(-\infty,\sqrt{-c}\right]$ when $\gamma$ extends to a geodesic ray $\tilde{\gamma}:\left[0,\infty\right)\to U$. Also, for every complex number $\lambda\neq\sqrt{-c}$, we have that the function $\cosh{\sqrt{-c}t}-\left(\sinh{\sqrt{-c}t}/\sqrt{-c}\right)\lambda\to\infty$ as $t\to\infty$, and then $J_{\gamma'(0)}(t)X_0\to\infty$ as $t\to\infty$ for any parallel vector field $X_t$ along $\gamma$ as long as $X_0\notin E_{\sqrt{-c}}^{C_{\gamma'(0)}}$. By equation \eqref{eq:shapeoperator}, we conclude that $A_{\xi_t}X_t\to0$ as $t\to\infty$. On the other hand, when $\lambda=\sqrt{-c}$, then $\cosh{\sqrt{-c}t}-\left(\sinh{\sqrt{-c}t}/\sqrt{-c}\right)\lambda=\cosh{\sqrt{-c}t}-\sinh{\sqrt{-c}t}\to0$ as $t\to\infty$. Hence, if $X_0\in E_{\sqrt{-c}}^{C_{\gamma'(0)}}$, we conclude from \eqref{eq:secondfundamentalform} and \eqref{eq:shapeoperator} that $\alpha_{\gamma(t)}\left(X_t,Y_t\right)\to\infty$ and $A_{\xi_t}X_t\to\infty$ as $t\to\infty$ when $\alpha_{\gamma(0)}\left(X_0,Y_0\right)\neq0$ and $A_{\xi_0}X_0\neq0$, respectively. The last assertion in part (\emph{b}) follows from the parallelism of $\ker A_{\xi_t}$ with respect to $\Delta$ along $\gamma$ and from
\begin{equation*}
\begin{aligned}
    \alpha_{\gamma(t)}\left(X_t,Y_t\right) &=\mathcal{P}_0^t\circ\alpha_{\gamma(0)}\left(J_{\gamma'(0)}(t)^{-1}X_0,Y_0\right)\\
    &=\mathcal{P}_0^t\circ\alpha_{\gamma(0)}\left(\frac{1}{\cosh\sqrt{-c}t-\sinh\sqrt{-c}t}X_0,Y_0\right).
\end{aligned}
\end{equation*}

If $\gamma$ extends to a geodesic line $\tilde{\gamma}:\left(-\infty,\infty\right)\to U$, then the real eigenvalues of $C_{\gamma'(0)}$ belong to the interval $\left(-\infty,\sqrt{-c}\right]\cap\left[-\sqrt{-c},\infty\right)=\left[-\sqrt{-c},\sqrt{-c}\right]$, and if either $\sqrt{-c}$ or $-\sqrt{-c}$ is not an eigenvalue of $C_{\gamma'(0)}$, then $\sqrt{-c}$ is not an eigenvalue of either $C_{\gamma'(0)}$ or $C_{-\gamma'(0)}$, hence $\alpha_{\gamma(t)}\to0$, or equivalently, $A_{\xi_t}\to0$ as either $t\to\infty$ or $t\to-\infty$, respectively. Case $c=0$ is analogous.
\end{proof}

Proposition \ref{fakemainlemma} will allow us to derive all the results in this work with little effort. Besides that, for complete submanifolds with relative nullity in the sphere, we have that $J_{\gamma'(0)}(0) = -J_{\gamma'(0)}(\pi)$ for any geodesic $\gamma$ as in Proposition \ref{fakemainlemma}. From this fact, we easily recover an important result due to Dajczer-Gromoll \cite{DGSphereNull}, which states that at any point where the index of relative nullity is minimal and for any normal direction at that point the numbers of positive and negative principal curvatures are equal. Notice also that Lemma 32 in Florit-Guimar\~{a}es \cite{belezinha} is embedded in part (\emph{i}) of Proposition \ref{fakemainlemma}.

The idea of the proof of Theorem \ref{cor:hirn} now comes down to finding a direction in which the splitting tensor is well behaved.

\begin{proof}[Proof of Theorem \ref{cor:hirn}]

Let $x_0\in M^n$ such that $\nu\left(x_0\right)=\nu_0$. Consider the linear map $L:\Delta\left(x_0\right)\oplus\textup{Skew}\left(\Delta^\perp\left(x_0\right)\right)\oplus\R\to\textup{End}\left(\Delta^\perp\left(x_0\right)\right)$ given by
\begin{equation*}
    L\left(T,\mathcal{S},\lambda\right)=C_T+\mathcal{S}+\lambda I.
\end{equation*}
Since $\nu_0\geq\nu_0^\perp\left(\nu_0^\perp+1\right)/2$, we conclude that $L$ cannot be injective. Hence, there exists a unit vector $T_0\in\Delta\left(x_0\right)$ such that $C_{T_0}\in\textup{Skew}\left(\Delta^\perp\left(x_0\right)\right)+\Span\{I\}$. Replacing $T_0$ with $-T_0$ if necessary, we can assume that $\sqrt{-c}$ is not an eigenvalue of $C_{T_0}$. Then $J_{T_0}(t)^{-1}\to0$ as $t\to+\infty$. By Proposition \ref{fakemainlemma}, we conclude that $\alpha\to0$ along the geodesic starting from $x_0$ in the direction $T_0$.
\end{proof}

Using the results in Florit \cite{FNonpositiveExtrinsic} we obtain Theorem \ref{cor:nonpositive-ex} almost as a corollary of Theorem \ref{cor:hirn}.

\begin{proof}[Proof of Theorem \ref{cor:nonpositive-ex}]
It follows from Florit's nullity estimate \cite{FNonpositiveExtrinsic} that $\nu_0\geq n-2p$. Thus $\nu_0^\perp\leq 2p$ and it is straightforward to check that the estimate $n\geq 2p^2+3p$ implies the one in Theorem \ref{cor:hirn}.
\end{proof}

The proof of Corollary \ref{cor:mean-curvature-vector} follows the same ideas as in the previous theorem.

\begin{proof}[Proof of Corollary \ref{cor:mean-curvature-vector}]
It follows from the Gauss equation that the scalar curvature $s$ of $M^n$ is given by
\begin{equation}
 s = -1 + \dfrac{n}{n-1} \left\|\mathcal{H}\right\|^2 - \dfrac{1}{n(n-1)} \left\| \alpha \right\|^2,  \label{scalar}
\end{equation}
where $\mathcal{H}$ is the mean curvature vector of $f$ and $|| \alpha ||$ denotes the norm of the second fundamental form, given by 
$$|| \alpha ||^2 = \displaystyle \sum_{i,j=1}^n || \alpha(e_i, e_j) ||^2,$$ for an orthonormal tangent basis $e_1, \ldots, e_n$. If the extrinsic sectional curvature is nonpositive, then $s \leq -1$ and consequently, by Equation \eqref{scalar}, we have
$$
\left\| \alpha \right\|^2 \geq n^2 \left\| \mathcal{H} \right\|^2.
$$
By Theorem \ref{cor:nonpositive-ex}, $f$ does not have extrinsic geometry bounded away from zero. Since $\left\| \mathcal{H} \right\|$ is constant, we must have $\left\| \mathcal{H} \right\|=0$.
\end{proof}

To prove Theorem \ref{theo:intcon} we will need the lemma below, which deals with the case where the conullity is totally geodesic and whose local version is well known (cf. Dajczer-Tojeiro \cite{newbook}).

\begin{lemma}\label{lem:cyl}
Let $f:M^n\to\mathbb{Q}_c^m$ be an isometric immersion of a complete Riemannian manifold and let $D$ be a smooth totally geodesic distribution of rank $0<k\leq n$ such that $D(x)\subset\Delta(x)$ for all $x\in M^n$. If the conullity distribution $D^\perp$ is totally geodesic, then $c=0$ and $f$ is a $k$-cylinder.
\end{lemma}

\begin{proof}
Since the splitting tensor $C$ of $D$ vanishes identically, it follows from \eqref{eq:spttns} that $c=0$. Because $D$ is totally geodesic, $D\subset\Delta$ and $M^n$ is complete, so are the leaves of $D$ and they are mapped by $f$ onto $k$-dimensional affine subspaces of $\R^m$, thus $L^{n-k}=M/D$ is a smooth manifold of dimension $n-k$ and the projection $\pi:M^n\to L^{n-k}$ is a vector bundle. The fact that the distribution $D^\perp$ is totally geodesic yields
\begin{equation}
    \tilde{\nabla}_Xf_*T=f_*\nabla_XT\in\Gamma(D)
\end{equation}
for all $X\in\Gamma\left(D^\perp\right)$ and $T\in\Gamma(D)$. Thus $V^k=f_*D$ is constant in $\R^m$, and the map $\varphi:M^n\to L^{n-k}\times V^k$ given by
\begin{equation*}
    \varphi(x)=\left(\pi(x),f(x)_V\right)
\end{equation*}
is a global trivialization of $\pi$, where $(\,)_V$ stands for the orthogonal projection onto $V$. Moreover, the map $\tilde{g}:M^n\to V^\perp$ given by $\tilde{g}(x)=f(x)_{V^\perp}$ projects to an immersion $g:L^{n-k}\to V^\perp$ whose induced metric on $L^{n-k}$ turns $\varphi$ into an isometry. Finally, it is straightforward to check that
\begin{equation*}
    f=\left(g\times I\right)\circ\varphi,
\end{equation*}
where $I:V\to V$ denotes the identity map.

\end{proof}

We are now able to understand the behavior of complete submanifolds with integrable conullity in space forms.

\begin{proof}[Proof of Theorem \ref{theo:intcon}]
Since $M^n$ is complete and $D$ is totally geodesic, the leaves of $D$ are also complete. Since $D^\perp$ is integrable, the splitting tensor $C_T$ is symmetric for all $T\in D$, thus part (\emph{i}) follows from part (\emph{i}) of Proposition \ref{fakemainlemma}. If $c=0$, part (\emph{ii.1}) of Proposition \ref{fakemainlemma} implies that $C=0$. Thus, $D^\perp$ is totally geodesic, and part (\emph{ii}) follows from Lemma \ref{lem:cyl}. Part (\emph{iii}) is immediate from part (\emph{ii.2}) of Proposition \ref{fakemainlemma}, since $A^i_T=C_T$ for every $T\in N_iL(x)=\Delta(x)$ and every $x\in L^{n-k}$.
\end{proof}

\section{Milnor's conjecture}\label{sec:milnor} 
In the 1960s, motivated by the then-recent result of Efimov \cite{MR0167938} on the nonexistence of complete Euclidean surfaces with negative Gaussian curvature $K$ bounded away from zero, J. Milnor suggested a conjecture concerning complete Euclidean surfaces whose $K$ does not change sign. It states that such a surface must have an umbilic, or else have points at which both principal curvatures simultaneously assume values arbitrarily close to zero.
\beco[Milnor \cite{MR0211332}]
Let $f:M^2\to\R^3$ be a complete umbilic-free surface with extrinsic geometry bounded away from zero. Then either the Gauss curvature $K$ of $f$ changes sign or else $K\equiv0$.
\enco
A small step toward the solution of the above conjecture was obtained by Smyth-Xavier \cite{MR914845}, where a consequence of their principal curvature theorem is that any complete surface of nonpositive Gauss curvature in $\R^3$ with one of its principal curvatures bounded away from zero must be a cylinder. Milnor-type results were also obtained for surfaces of $\mathbb{H}^3$ and $\Sp^3$ by G\'{a}lvez-Mart\'{i}nez-Teruel \cite{MR3351996}, where the authors use an abstract result about Codazzi pairs to prove that no complete surface can be immersed into $\mathbb{H}^3$ or $\Sp^3$ with nonpositive extrinsic or intrinsic sectional curvature, respectively, and one of its principal curvatures bounded away from zero.

It is interesting to investigate a similar question for hypersurfaces of space forms. Note that the same principal curvature theorem of Smyth-Xavier \cite{MR914845} coupled with Hartman-Niremberg theorem \cite{MR0126812} actually implies that every complete hypersurface in $\R^{n+1}$ with nonpositive sectional curvature and one of its principal curvatures bounded away from zero is necessarily a cylinder over a plane curve. In $\Sp^{n+1}$, Lemma 32 of Florit-Guimar\~{a}es \cite{belezinha} yields that any complete hypersurface with nonpositive extrinsic sectional curvature must be totally geodesic for $n\geq4$. On the other hand, if $n=3$ the polar map $\psi:N^1_f\Sp^2_{1/3}\to\Sp^4$ of the Veronese surface $f:\Sp^2_{1/3}\to\Sp^4$, given by
\begin{equation*}
    \psi\lp y,w\rp=w,
\end{equation*}
is a minimal isoparametric hypersurface with three distinct principal curvatures and index of relative nullity $\nu\equiv1$. Thus, $\psi$ gives a nontotally geodesic counterexample of a complete umbilic-free hypersurface in $\Sp^4$ with nonpositive extrinsic sectional curvature and shape operator bounded away from zero.

Finally, concerning hypersurfaces in $\mathbb{H}^{n+1}$, Theorem \ref{cor:nonpositive-ex} in the case $p=1$ can be restated as follows.

\begin{corollary}\label{hhmcs}
Let $f:M^n\to\mathbb{H}^{n+1}$ be a complete hypersurface with nonpositive extrinsic sectional curvature, $n\geq5$. Then $f$ does not have extrinsic geometry bounded away from zero.
\end{corollary}

For $n=3,4$, it follows from part (\emph{ii}) of Proposition \ref{fakemainlemma} that no principal curvature can be bounded away from zero. In the case $n=4$, we can actually implement a workaround to show that Corollary \ref{hhmcs} is still true, although the proof uses distinct techniques involving the Gauss parametrization and will not be presented in this article.

If we assume instead that $f$ has nonnegative extrinsic sectional curvature in the above statement, then plenty of counterexamples arise from a hyperbolic cylinder $f:\Sp^k\left(\rho\right)\times\Hy^{n-k}\left(\sqrt{1+\rho^2}\right)\hookrightarrow\Hy^{n+1}$, which is an isoparametric hypersurface of $\Hy^{n+1}$ with two distinct principal curvatures. We have that $f$ is clearly a complete umbilic-free hypersurface with shape operator bounded away from zero and strictly positive extrinsic sectional curvature. Since all these properties are stable under small perturbations, we can construct many more generic counterexamples from $f$.

\bibliographystyle{acm}
\bibliography{bib-completesubmanifolds}
\end{document}